\documentclass[12pt]{extarticle}
\usepackage{amsmath, amsthm, amssymb, color}
\usepackage[colorlinks=true,linkcolor=blue,urlcolor=blue]{hyperref}
\usepackage{graphicx}
\usepackage{caption}
\usepackage{mathtools}
\usepackage{enumerate}
\usepackage{verbatim}
\usepackage{tikz,tikz-cd,tikz-3dplot}
\usepackage{amssymb}
\usetikzlibrary{matrix}
\usetikzlibrary{arrows}
\usepackage{algorithm}
\usepackage[noend]{algpseudocode}
\usepackage{caption}
\usepackage[normalem]{ulem}
\usepackage{subcaption}
\oddsidemargin \evensidemargin
\usepackage{makecell}
\usepackage{array}
\usepackage[giveninits=true, maxbibnames=99, style=numeric]{biblatex}
\DeclareFieldFormat{pages}{#1}
\renewbibmacro{in:}{}

\DeclareFieldFormat[article]{title}{\textit{#1}}
\DeclareFieldFormat[article]{journaltitle}{\normalfont{#1}}
\DeclareFieldFormat[article]{volume}{\textbf{#1}}
\DeclareListFormat[misc]{date}{\mkbibparens{#1}}
\newtheorem{theorem}{Theorem}
\newtheorem*{theorem*}{Theorem}
\newtheorem{proposition}[theorem]{Proposition}
\newtheorem{lemma}[theorem]{Lemma}
\newtheorem{corollary}[theorem]{Corollary}

\theoremstyle{definition}

\newtheorem{conjecture}[theorem]{Conjecture}

\newtheorem{example}[theorem]{Example}

\numberwithin{theorem}{section}

\newcommand{\PP}{\mathbb{P}}
\newcommand{\RR}{\mathbb{R}}

\newcommand{\CC}{\mathbb{C}}

\newcommand{\NN}{\mathbb{N}}

\title{\bf Likelihood Geometry \\ of the Squared Grassmannian}

\author{Hannah Friedman}

\date{}
\begin{document}
\maketitle

\begin{abstract}
  \noindent
  We study projection determinantal point processes and their connection to the squared Grassmannian.
  We prove that the log-likelihood function of this statistical model has $(n - 1)!/2$ critical points, all of which are real and positive, thereby settling a conjecture of Devriendt, Friedman, Reinke, and Sturmfels.
\end{abstract}

\let\thefootnote\relax\footnote{\noindent\emph{MSC2020:} 62R01, 14M15, 14P25.}

\section{Introduction}
Determinantal point processes (DPPs) on a finite ground set are a family of statistical models whose state space is the power set of the ground set.
The key feature of DPPs is negative correlation, meaning that DPPs select for diverse subsets of the ground set, for some notion of diversity.  
DPPs elegantly model negative correlation and consequently have many applications in physics, random matrix theory, and machine learning \cite{BMRU, KL}.

DPPs were previously studied from an algebraic statistics perspective in \cite{FSZ}.
We continue in this vein, motivated by the connection between the Grassmannian and a special class of DPPs called projection DPPs which is studied by Devriendt, Friedman, Reinke, and Sturmfels in \cite{DFRS}.
We go beyond the work in \cite{FSZ}, giving a complete picture of the likelihood geometry of projection DPPs.
Our main result is the resolution of \cite[Conjecture 4.2]{DFRS}.

A DPP is a random variable $Z$ on the power set $2^{[n]}$ of the finite set $[n] = \{1, \ldots, n\}$ such that
$\PP[I \subseteq Z] = \det(K_I)$, 
where $K$ is an $n \times n$ symmetric matrix with eigenvalues in $[0,1]$ and $K_I$ is the principal submatrix of $K$ whose rows and columns are indexed by $I$.
The matrix $K$ is the \emph{kernel} of the DPP and encodes the similarity of the elements of the ground set $[n]$.

We study maximum likelihood estimation for determinantal point processes.
Given some data $(u_I)_{I \in 2^{[n]}}$, we seek the maximizer of the log-likelihood function
\begin{equation}\label{eqn:implicitlikelihood}
  L_u(q) = \sum_{I \in 2^{[n]}} u_I \log(q_I) -  \big ( \sum_{I \in 2^{[n]}} u_I \big ) \log\big ( \sum_{I \in 2^{[n]}} q_I \big)
\end{equation}
where the vector $q = (q_I)_{I \in 2^{[n]}}$ is a probability distribution coming from a DPP, up to scaling.
The maximizer of $L_u$, called the \emph{maximum likelihood estimate} (MLE), is the point on the DPP model that best explains the data $u$. 
The number of complex critical points of (\ref{eqn:implicitlikelihood}) for generic data $u$ is the \emph{maximum likelihood degree} (ML degree) of the model.
The ML degree gives an algebraic bound on the complexity of maximum likelihood estimation.

To count the critical points of (\ref{eqn:implicitlikelihood}), we need a precise description of the constraints on $(q_I)_{I \in 2^{[n]}}$.
In particular, we need to know the probability $q_I$ of observing a set $I$, rather than $\PP[I \subseteq Z]$.
We discuss two scenarios in which there are nice formulae for $q_I$.

When the eigenvalues of $K$ lie in $(0, 1)$, there exists a positive definite matrix $\Theta$ such that $q_I = \frac{\det(\Theta_I)}{\det(\Theta + \mathrm{Id}_n)}$; see \cite{BMRU} for details.
The maximum likelihood estimation problem can then be solved by maximizing (\ref{eqn:implicitlikelihood}) subject to the constraint that $q$ satisfies the hyperdeterminantal relations; see \cite{oeding}. 
Practitioners use the unconstrained, parametric log-likelihood function
\begin{align*}
  L_u(\Theta) = \sum_{I \in 2^{[n]}} u_I \log(\det(\Theta_I)) -  \big ( \sum_{I \in 2^{[n]}} u_I \big ) \log(\det(\Theta + \mathrm{Id}_n)).
\end{align*}
In \cite{FSZ}, the authors express the number of critical points of the parametric log-likelihood function in terms of the ML degree of the hyperdeterminantal variety.

In this paper, we take the approach of \cite[Section 4]{DFRS} and restrict the eigenvalues of $K$ to be in $\{0, 1\}$.
If $K$ has rank $d$, then $K$ is the unique orthogonal projection matrix onto a $d$-dimensional subspace and we denote $K$ by $P$.
DPPs whose kernels are projection matrices are called \emph{projection DPPs} \cite{KL}.
By \cite[Corollary 4.3]{DFRS}, the state space of a projection DPP with kernel $P$ is the set of $d$-subsets of $[n]$, denoted $\binom{[n]}{d}$, and the probability of observing a $d$-subset $I$ is $q_I = \det(P_I)$.
The column space of $P$ is a $d$-dimensional subspace of $\RR^n$ and can be represented by its vector of Pl\"{u}cker coordinates $(p_I)_{I \in \binom{[n]}{d}}$; see \cite[Chapter 5]{MS}.
There is a particularly nice relationship between the principal minors of $P$ and the Pl\"{u}cker coordinates of its column span. 
\begin{lemma}[\cite{DFRS}, Lemma 3.1]\label{lem:twolives}
  The principal minors of an orthogonal projection matrix $P$ are proportional to the corresponding squared Pl\"{u}cker coordinates:
  \begin{align*}
    q_I = \det(P_I) = \frac{p^2_I}{\sum_{J \in \binom{[n]}{d}} p_J^2}.
  \end{align*}
\end{lemma}
An immediate consequence of Lemma~\ref{lem:twolives} is that every probability distribution arising from a projection DPP can be written as the entry-wise square of a vector of Pl\"{u}cker coordinates.
This observation motivates the definition of the \emph{squared Grassmannian}, denoted $\mathrm{sGr}(d,n)$, as the Zariski closure of the  model for projection DPPs.
The squared Grassmannian is the image of the squaring map from the Grassmannian ${\rm Gr}(d, n) \subset \mathbb{P}^{\binom{n}{d} - 1}$ in its Pl\"{u}cker embedding given by $(p_I)_{I \in \binom{[n]}{d}} \mapsto (p_I^2)_{I \in \binom{[n]}{d}}$; see \cite{AV, DFRS}.
Thus the MLE of a projection DPP is found by maximizing the implicit log-likelihood function (\ref{eqn:implicitlikelihood}) over  ${\rm sGr}(d, n)$.

We mainly focus on the case where $d = 2$, so our model has state space $\binom{[n]}{2}$.
We determine the number of critical points of the implicit log-likelihood function:
\begin{theorem}[\cite{DFRS}, Conjecture 4.2]\label{thm:main}
  The ML degree of ${\rm sGr}(2,n)$ is $\frac{(n - 1)!}{2}$ for $n \geq 3$.
\end{theorem}
In Section~\ref{sec:realsolutions}, we prove that all of these critical points are relevant from a statistical perspective.
This is reminiscent of Varchenko's Theorem \cite[Theorem 13]{CHKS}. 
\begin{theorem}\label{thm:nonnegative}
  For $u \in \NN^{\binom{n}{2}}$, each critical point of the log-likelihood function~(\ref{eqn:implicitlikelihood}) over ${\rm sGr}(2,n)$ is real, positive, and a local maximum when restricted to real points of ${\rm sGr}(2,n)$.
\end{theorem}

Another consequence of Lemma~\ref{lem:twolives}, is that the model of projection DPPs is parameterized by $d \times n$ matrices $M_{d,n}$:
\begin{equation}\label{eqn:parameterization}
  q_I = \frac{(M_{d,n})_I^2}{ \sum_{J \in \binom{[n]}{d}} (M_{d,n})_J^2}
\end{equation}
where $(M_{d,n})_I$ is the maximal minor of $M_{d,n}$ whose columns are indexed by $I$.
We then have the parametric log-likelihood function for projection DPPs:
\begin{equation}\label{eqn:parametriclikelihood}
  L_u(M_{d,n}) = \sum_{I \in \binom{[n]}{d}} u_I \log(\det((M_{d,n})_I)^2) - 
  \big ( \sum_{I \in \binom{[n]}{d}} u_I \big ) \log\big ( \sum_{I \in \binom{[n]}{d}} \det((M_{d,n})_I)^2 \big).
\end{equation}
If we restrict to the matrices $M_{d,n}$ whose first $d \times d$ square is the identity, we get a $2^{n-1}$-degree parameterization; see \cite[Section 4]{DFRS}.
Therefore the number of critical points of (\ref{eqn:parametriclikelihood}) is $2^{n-1}$ times the ML degree of the squared Grassmannian.
\begin{example}[$d = 2, n = 3$]\label{ex:basecase}
  The implicit log-likelihood function is
  \begin{align*}
    L_u(q) = u_{12} \log(q_{12}) + u_{13} \log(q_{13}) + u_{23} \log(q_{23}) - (u_{12} + u_{13} + u_{23}) \log(q_{12} + q_{13} + q_{23}).
  \end{align*}
  Because $\rm{sGr}(2,3) = \PP^2$, the MLE is $u$ and the ML degree is one.
  We can find the same answer parametrically: if $M_{2,3} = \begin{pmatrix} 1 & 0 & x_3 \\ 0 & 1 & y_3  \end{pmatrix}$, then 
  \begin{align*}
    L_u(M_{2,3}) = u_{12} \log(1) + u_{13} \log(y_3^2) + u_{23} \log(x_3^2) - (u_{12} + u_{13} + u_{23}) \log(1 + x_3^2 + y_3^2)
  \end{align*}
  whose partial derivatives are
  \begin{align*}
    \frac{\partial L_u}{\partial x_3} = \frac{2u_{23}}{x_3} - \frac{2(u_{12} + u_{13} + u_{23})x_{3}}{1 + x_3^2 + y_3^2} = 0, &&
    \frac{\partial L_u}{\partial y_3} = \frac{2u_{13}}{y_3} - \frac{2(u_{12} + u_{13} + u_{23})y_{3}}{1 + x_3^2 + y_3^2} = 0.
  \end{align*}
  We can solve these equations by hand to find $x_3 = \pm \sqrt{\frac{u_{23}}{u_{12}}}$ and $y_3 = \pm \sqrt{\frac{u_{13}}{u_{12}}}$ for a total of four solutions.
  Our parameterization sends the parametric solution to the implicit one:
  \begin{align*}
    \begin{pmatrix}
      1 & 0 & \pm \frac{\sqrt{u_{23}}}{\sqrt{u_{12}}}\\
      0 & 1 & \pm \frac{\sqrt{u_{13}}}{\sqrt{u_{12}}}.
    \end{pmatrix}
    \mapsto (1: \frac{u_{13}}{u_{12}} : \frac{u_{23}}{u_{12}}).
  \end{align*}
\end{example}

Likelihood inference on the Grassmannian has been studied in other contexts as well; see \cite[Section 4]{DFRS} for an overview and comparison of three different models.
Of particular interest is the configuration space of $X(d, n)$ which is modeled by the Grassmannian modulo a torus action.
When $d = 2$, $X(2, n)$ is the moduli space $\mathcal M_{0,n}$ of $n$ marked points in $\PP^1$.
Similar to our situation, the number of critical points of the likelihood function on $\mathcal M_{0,n}$ grows exponentially and all $(n - 3)!$ critical points are real \cite[Proposition 1]{ST}.

This paper is organized as follows.
In Section~\ref{sec:setup}, we outline the approach for the proof of Theorem~\ref{thm:main} and study the topology of the parametric model arising from (\ref{eqn:parameterization}).
In Section~\ref{sec:stratification}, we study the stratification of the parametric model by a deletion map and prove Theorem~\ref{thm:main}.
In Section~\ref{sec:realsolutions}, we turn to real critical points and prove Theorem~\ref{thm:nonnegative}.
We finish by computing the MLE for random data for $4\leq n\leq 9$ and we record the runtimes.

Code for the numerical experiments in Section~\ref{sec:realsolutions} 
is available 
in the \verb+MathRepo+ collection via \url{https://mathrepo.mis.mpg.de/likelihood_geometry_squared_grassmannian}.

\section{Parametric Model}\label{sec:setup}
Our approach to computing the ML degree of ${\rm sGr}(2,n)$ is topological. 
If $X \subseteq \PP^n$ is a projective variety, then $\mathcal H = \{(x_0 : \cdots : x_n) \in X \mid (\sum^n_{i=0}x_i) \cdot \prod_{i=0}^n x_i = 0\}$ is the set of points of $X$ where the log-likelihood function is undefined.
The set $X \backslash \mathcal H \subseteq (\CC^*)^{n+1}$ is a \emph{very affine variety}, meaning that it is a closed subvariety of the torus $(\CC^*)^{n+1}$.
We use the following theorem, which first appeared as \cite[Theorem 19]{CHKS} under stronger hypotheses. 
\begin{theorem}[\cite{huh}, Theorem 1]\label{thm:mleuler}
  If the very affine variety $X \backslash \mathcal H$ is smooth of dimension $d$, then the ML degree of $X$ equals the signed Euler characteristic $(-1)^d\chi(X \backslash \mathcal H)$.
\end{theorem}
The Euler characteristic is a topological invariant that can be defined as the alternating sum of the Betti numbers of the space.
We will often make use of the excision property of the Euler characteristic: if $Z = X \sqcup Y$, then $\chi(Z) = \chi(X) + \chi(Y)$.
We also rely on the fibration property of the Euler characteristic; if $f: X \to Y$ is a fibration, then $\chi(X) = \chi(Y) \chi(F)$, where $F$ is the fiber of a point in $Y$.
For the maps we consider, proving that the fibers of a map are homeomorphic suffices to show that the map is a fibration.

We apply Theorem~\ref{thm:mleuler} to the squared Grassmannian.
The \emph{open squared Grassmannian} is the very affine variety
\begin{align*}
  {\rm sGr}(d,n)^\circ = \big \{(q_I)_{I \in \binom{[n]}{d}} \in {\rm sGr}(d, n) \mid
  \big ( \sum_{I \in \binom{[n]}{d}} q_I \big ) \cdot \prod_{I \in \binom{[n]}{d}} q_{I} \neq 0 \big \} \subseteq (\CC^*)^{\binom{n}{d}}.
\end{align*}
The variety ${\rm sGr}(d,n)^\circ$ is smooth because ${\rm Gr}(d,n)$ is smooth and the Jacobian of the squaring map has full rank on points in the preimage of ${\rm sGr}(d,n)^\circ$. 
Because $\dim({\rm sGr}(d,n)) = d(n-d)$, by Theorem~\ref{thm:mleuler}, the ML degree of $\mathrm{sGr}(d,n)$ is $(-1)^{d(n-d)} \chi({\rm sGr}(d,n)^\circ)$.

We now turn to the case $d = 2$ and parameterize the open squared Grassmannian by 
\begin{equation}\label{eqn:2parameterization}
  X_n \to {\rm sGr}(2,n)^\circ, \ \ 
  M_n = 
  \begin{pmatrix}
    1 & 0 & x_3 & x_4 & \cdots & x_n \\
    0 & 1 & y_3 & y_4 & \cdots & y_n
  \end{pmatrix}
  \mapsto
  (p_{ij}^2)_{1 \leq i < j \leq n}
\end{equation}
where $p_{ij}$ is the $2$-minor formed from columns $i$ and $j$ of $M_n$ and $X_n$ is the subset of $\CC^{2(n-2)}$ such that the image of the parameterization is ${\rm sGr}(2,n)^\circ$.
To be specific, we define
\begin{align*}
  Q_n = \sum_{1 \leq i < j \leq n} p_{ij}^2 && \textrm{and} && 
  X_n = \big \{
  M_n
  \in \mathbb{C}^{2(n-2)}
    \colon Q_n \cdot \big ( \prod_{1 \leq i< j \leq n} p_{ij} \big ) \neq 0 
  \big \}.
\end{align*}
We can explicitly write the parametric log-likelihood function in the $d = 2$ case as
  \begin{equation}\label{eqn:2parametriclikelihood}
    L_u(M_n) = \sum_{i = 3}^n  \big (u_{1i}\log(y_i^2) + u_{2i}\log(x_i^2)\big ) + \sum_{\mathclap{3 \leq i < j \leq n}}u_{ij} \log((x_iy_j - y_ix_j)^2)  - \sum_{\mathclap{1 \leq i < j \leq n}} u_{ij} \log(Q_n).
  \end{equation}
  The map in (\ref{eqn:2parameterization}) has degree $2^{n-1}$ with no ramification points because we can flip the signs of any row and column of the matrix $M_n$ independently; see \cite[Remark 4.4]{DFRS}.
  By multiplicativity of the Euler characteristic, the implicit and parametric models are related by:
\begin{equation}\label{eqn:powertwo}
  \chi(X_n) = 2^{n-1} \chi({\rm sGr}(2, n)^\circ) = 2^{n-1}{\rm ML Deg}({\rm sGr}(2,n))  = \#\{ \textrm{critical points of (\ref{eqn:2parametriclikelihood})}\}.
\end{equation}
By (\ref{eqn:powertwo}), we can compute the number of critical points of both the implicit and parametric log-likelihood functions from $\chi(X_n)$.
We compute $\chi(X_n)$ in Section~\ref{sec:stratification} by an argument similar to that in \cite{ABFKSTL, CHKO, EGPSY}.
In particular, we wish to define a map $X_{n + 1} \to X_n$ which deletes the last column and use the fact that this map is a stratified fibration to compute the Euler characteristic inductively. 
However, this deletion map is not well defined, because $Q_n$ vanishes on points in $X_{n+1}$, so we define this map only on the open subset of $X_{n + 1}$ where $Q_n$ does not vanish, denoted $  X_{n + 1}^\circ  = \left \{  M_{n+1} \in X_{n + 1} \colon Q_n \neq 0  \right \}.$
We define the~projection
\begin{align*}
  \pi_{n + 1}: X_{n + 1}^\circ \to X_n, && 
  \begin{pmatrix}
    1 & 0 & x_3 & x_4 & \cdots & x_n & x_{n+1}\\
    0 & 1 & y_3 & y_4 & \cdots & y_n & y_{n+1}
  \end{pmatrix}
  \mapsto
  \begin{pmatrix}
    1 & 0 & x_3 & x_4 & \cdots & x_{n}\\
    0 & 1 & y_3 & y_4 & \cdots & y_{n}
  \end{pmatrix}
\end{align*}
and argue that the change from $X_{n + 1}$ to $X_{n + 1}^\circ$ does not affect the Euler characteristic.
To do so, we observe that the fiber of a point $M_n \in X_n$ under $\pi_{n + 1}$ is the complement in $(\CC^*)^2$ of the lines $p_{i(n + 1)}$ and conic $Q_{n + 1}$. 
The lines $p_{i(n+1)}$ all pass through the origin, so the key to understanding the geometry of the fibers is to understand the conic $Q_{n+1}$.
We can write
\begin{align*}
  Q_{n + 1} = \begin{pmatrix}1 & x_{n+1} & y_{n+1} \end{pmatrix}   \begin{pmatrix}
    Q_{n} & 0 & 0\\
    0 & 1 + \sum_{i = 3}^{n} y_i^2 & - \sum_{i = 3}^{n} x_iy_i\\
    0 & - \sum_{i = 3}^{n} x_iy_i & 1 +	\sum_{i = 3}^{n} x_i^2
  \end{pmatrix}
 \begin{pmatrix}1  \\x_{n+1} \\ y_{n+1} \end{pmatrix}.
\end{align*}
Taking the determinant of the matrix, we find that the discriminant of $Q_{n + 1}$ as a conic in $x_{n + 1}$ and $y_{n + 1}$ is $Q_n^2$.
We  prove that restricting to $X_{n + 1}^\circ$ does not change the Euler characteristic.
\begin{lemma}\label{lem:complement}
  For $n \geq 3$, we have $\chi(X_{n+1}) = \chi(X_{n + 1}^\circ)$. 
\end{lemma}
\begin{proof}
  By the excision property, $\chi(X_{n+1}) = \chi(X_{n + 1}^\circ) +\chi(X_{n+1} \backslash X_{n + 1}^\circ)$.
  We shall prove that $\chi(X_{n+1} \backslash X_{n + 1}^\circ) = 0$ by showing that the map below is a fibration on its image:
  \begin{align*}
    X_{n + 1} \backslash X_{n + 1}^\circ \to (\CC^*)^{2(n-2)}, &&
    \begin{pmatrix}
     1 & 0 & x_3 & x_4 & \cdots & x_n & x_{n+1}\\
     0 & 1 & y_3 & y_4 & \cdots & y_n & y_{n+1}
  \end{pmatrix}
  \mapsto
  \begin{pmatrix}
    1 & 0 & x_3 & x_4 & \cdots & x_{n}\\
    0 & 1 & y_3 & y_4 & \cdots & y_{n}
  \end{pmatrix}.
  \end{align*}
  It suffices to show that the fibers  are homeomorphic. 
  Every nonempty fiber $F$ is the complement in $(\CC^*)^2$ of $n - 1$ lines through the origin: each minor $p_{i(n + 1)}$ for $i \geq 3$ contributes one line and  $Q_{n + 1}$ degenerates into a double line since $Q_n = 0$.
  The Euler characteristic of finitely many lines through the origin in $(\CC^*)^2$ is zero because each line is a $\PP^1$ with two points removed and the lines do not intersect. 
  Because $\chi((\CC^*)^2) = 0$, the Euler characteristic of the complement of $n -1$ lines through the origin in $(\CC^*)^2$ is zero, too.
  By the fibration property of the Euler characteristic and because $\chi(F) = 0$, we conclude that $\chi(X_{n + 1} \backslash X_{n + 1}^\circ) = 0$.
\end{proof}

\section{Stratification of the Parametric Model}\label{sec:stratification}
We now prove that the map $\pi_{n + 1}$ is a stratified fibration and compute $\chi(X_n)$.
A map $f: X \to Y$ between complex algebraic varieties with a stratification $\{S_\alpha\}_{\alpha=1}^m$ of $Y$ by closed sets is a
\emph{stratified fibration} if the restriction of $f$ to each open stratum $S^\circ_\alpha = S_\alpha \backslash \bigcup_{S_\beta \subsetneq S_\alpha} S_\beta$ is a fibration with fiber denoted $F_{S_\alpha}$.
We use M\"{o}bius inversion with the fibration property of the Euler characteristic, to compute Euler characteristics along stratified fibrations.
\begin{lemma}[\cite{ABFKSTL}, Lemma 2.3]\label{lem:combinatorics}
  Let $f: X \to Y$ be a stratified fibration.
  Let $\mathcal{S}$ be the poset of closed strata $S_\alpha$ of $f$, let $F_{S_\alpha}$ be the fiber of a generic point in $S_\alpha$, and let $\mu$ be the M\"{o}bius function of $\mathcal S$.
  Then
  \begin{align*}
    \chi(X) 
    &= \chi(Y) \cdot \chi(F_Y) + \sum_{S_\alpha \in \mathcal S} \chi(S_\alpha) \cdot \sum_{S_\beta \in \mathcal S, S_\beta \supseteq S_\alpha} \mu(S_\alpha, S_\beta) \cdot (\chi(F_{S_\beta}) - \chi(F_Y)).
  \end{align*}
\end{lemma}
In our situation, only the generic fiber contributes: the sums vanish and $\chi(X) = \chi(Y)~\chi(F_Y)$. 
We first describe the stratification of $X_n$ and the fibers of $\pi_{n + 1}$.
\begin{lemma}\label{lem:stratification}
  The map $\pi_{n + 1}$ is a stratified fibration with stratification
  \begin{align*}
    \mathcal S = \{ X_n \} \cup \{S_i \mid 1 \leq i \leq n\} \cup \{S_i \cap S_j \mid 1 \leq i < j \leq n\},
  \end{align*}
  where $S_i = \big\{M_n \in X_n \mid \sum_{j = 1}^{n} p_{ij}^2 = 0\big\}$.
  The fibers of $\pi_{n + 1}$ are the complements of $n$ lines $p_{i(n+1)}$ and a conic $Q_{n + 1}$ in $\CC^2$ where the intersections of the lines with the conic are
  \begin{center}
    \begin{tabular}{ll}
      $Q_{n + 1} \cap p_{k(n+1)} = \textrm{two points}$ for all $k$ & in  $F_{X_n}$.\\
      $Q_{n + 1} \cap p_{k(n+1)} = \begin{cases} \textrm{two points for $k \neq i$}\\ \emptyset \textrm{ for $k = i$}.  \end{cases}$ & in $F_{S_i}$.\\
      $Q_{n + 1} \cap p_{k(n+1)} = \begin{cases} \textrm{two points for $k \neq i,j$}\\ \emptyset \textrm{ for $k \in \{i,j\}$} \end{cases}$ & in $F_{S_i \cap S_j}$.
    \end{tabular}
  \end{center}
\end{lemma}
\begin{proof}
  To prove $\pi_{n + 1}$ is a fibration, it suffices to show that the fibers of $\pi_{n + 1}$ are homeomorphic on each open stratum $S^\circ$.
  The fiber in $\CC^2$ of a point $M_n \in X_n$ is the complement of the lines $p_{i(n + 1)}$ and the conic $Q_{n + 1}$.
  Since all the lines $p_{i(n + 1)}$ pass through the origin, their intersection is the same for any $M_n \in X_n$.
  The variation in the fibers comes from the intersection with the conic. 
  We claim that every line $p_{i(n+1)}$ either intersects the conic $Q_{n + 1}$ in two distinct points or
  not at all.
  In the intersection locus of $p_{i(n+1)} = 0$ and $Q_{n + 1} = 0$, 
  \begin{equation}\label{eqn:stratproof}
    Q_{n + 1} = Q_n + \sum_{j = 1}^n p_{j (n + 1)}^2 = Q_n + \frac{x_{n+1}^2}{x_i^2} \sum_{j = 1}^n p_{ij}^2 = 0.
  \end{equation}
  If $\sum_{j = 1}^n p_{ij}^2 \neq 0$, then there are two distinct intersection points.
  If $\sum_{j = 1}^n p_{ij}^2 = 0$, then (\ref{eqn:stratproof}) becomes $Q_n = 0$, which is a contradiction, so the intersection of $Q_{n + 1}$ and $p_{i(n + 1)}$ is empty.
  Thus, in the fiber $F_{X_n}$, every line intersects the conic in two points and in the fiber $F_{S_i}$, every line intersects the conic in two points, except $p_{i(n+1)}$ which does not intersect the conic. 
  
  It remains to be shown that there can be at most two lines $p_{i(n+1)}$ which do not intersect $Q_{n + 1}$.
  When $p_{i(n+1)}$ does not intersect $Q_{n + 1}$ in $\CC^2$, the line and conic intersect at a double point at infinity.
  In other words, $p_{i(n+1)}$ is tangent to the projectivization of $Q_{n + 1}$ at infinity.
  In projective space, $p_{i(n + 1)}$ intersects $Q_{n + 1}$ tangentially if and only if $p_{i(n + 1)}^\vee$ lies in the dual conic $Q_{n + 1}^\vee$, i.e., $\begin{pmatrix} 0 & -y_i & x_i \end{pmatrix} A_{n}^{-1} \begin{pmatrix} 0 & -y_i & x_i \end{pmatrix}^T =\sum_{j = 1}^n p_{ij}^2 = 0$.
  Since at most two lines passing through the origin can be tangent to a conic, only two of the lines $p_{i(n + 1)}$ can have an empty intersection with the conic $Q_{n + 1}$ at once. 
  Thus, in the fiber $F_{S_i \cap S_j}$, every line intersects the conic in two points, except $p_{i(n + 1)}$ and $p_{j(n + 1)}$ which do not intersect the conic. 
\end{proof}

The combinatorics of the stratification $\mathcal S$ is simple, which allows us to evaluate the formula for the Euler characteristic in Lemma~\ref{lem:combinatorics}.
To apply the lemma, we need the M\"{o}bius function of the poset of each closed stratum; see Figure~\ref{fig:posets} for the posets of each of the special strata along with the values of their M\"{o}bius functions.
The remaining computations necessary to prove Theorem~\ref{thm:main} are Euler characteristics of the fibers and of the strata.
\begin{figure}[h]
  \begin{center}
  \begin{tikzpicture}[scale=0.5]
    \node (xn) at (0,0) {$X_n$};
    \node (si) at (0, -4) {$S_i$};
    \draw (xn) -- (si);
    \node (xnlab) at (1, 0) {\color{blue}{$-1$}};
    \node (silab) at (1, -4) {\color{blue}{$1$}};    
  \end{tikzpicture}
  \hspace{3 cm}
  \begin{tikzpicture}[scale=0.5]
    \node (xn) at (0,0) {$X_n$};
    \node (si) at (-2, -2) {$S_i$};
    \node (sj) at (2, -2) {$S_j$};
    \node (sij) at (0, -4) {$S_i \cap S_j$};
    \draw (xn) -- (sj);
    \draw (xn) -- (si);
    \draw (sij) -- (sj);
    \draw (sij) -- (si);
    \node (xnlab) at (1, 0) {\color{blue}{$1$}};
    \node (silab) at (-3, -2) {\color{blue}{$-1$}};
    \node (sjlab) at (3, -2) {\color{blue}{$-1$}};
    \node (sijlab) at (1.5, -4) {\color{blue}{$1$}};
  \end{tikzpicture}
  \end{center}
  \caption{Posets for the special strata. The values of the M\"{o}bius function $\mu(-, S_i)$ and $\mu(-, S_i \cap S_j)$, respectively, are shown in blue.}\label{fig:posets}
\end{figure}
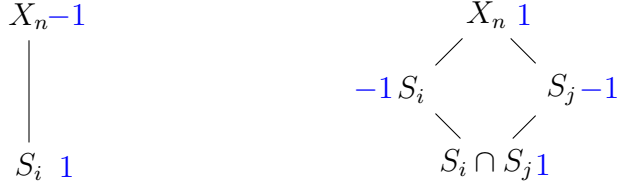
\begin{lemma}[Euler characteristics of fibers] \label{lem:fibers}
  The Euler characteristics of the fibers are
  \begin{align*}
    \chi(F_{X_n}) = 2n && \chi(F_{S_i}) = 2n - 2 && \chi(F_{S_i \cap S_j}) = 2n - 4 && \textrm{for $1 \leq i < j \leq n$.}
  \end{align*}
\end{lemma}
\begin{proof}
  In $\CC^2\backslash \{(0,0)\}$, we have $\chi(Q_{n + 1}) = 0$, since the conic $Q_{n + 1}$ has nonzero constant term $Q_n$ and is therefore topologically equivalent to a $\PP^1$ with two points at infinity removed. 
  Similarly, in $\CC^2\backslash \{(0,0)\}$, we have $\chi(p_{i(n + 1)}) = 0$, since the line $p_{i(n + 1)}$ is topologically equivalent to a $\PP^1$ with the origin and a point at infinity removed.
  Suppose that $M_n \in X_n$.
  Because $\chi(\CC^2\backslash \{(0,0)\}) = 0$, by the excision property,
  \begin{align*}
    \chi(\pi^{-1}(M_n)) &= - \chi(Q_{n + 1}) - \sum_{i = 1}^n \chi(p_{i(n + 1)}) + \sum_{i = 1}^n \chi(p_{i(n + 1)} \cap Q_{n + 1})\\
    &= \sum_{i = 1}^n \chi(p_{i(n + 1)} \cap Q_{n + 1}) = \left | \bigcup_{i = 1}^n (p_{i(n + 1)} \cap Q_{n + 1}) \right |
  \end{align*}
  where the last equality holds because the intersection of a line with a conic is zero dimensional.
  We computed these intersections in Lemma~\ref{lem:stratification}.
  In the generic fiber, each of the $n$ lines intersects the conic in two points, for a total of $2n$ intersection points.
  The intersections for $S_i, S_i \cap S_j$ are the same, but with $n - 1$ and $n - 2$ lines, respectively. 
\end{proof}

We now show that the Euler characteristics of the closed, codimension one strata vanish.

\begin{lemma}[Euler characteristics of strata]\label{lem:strata}
  We have $\chi(S_i) = 0$ for all $i$. 
\end{lemma}
\begin{proof}
  It is sufficient to prove the claim for $S_n$ by symmetry.
  We first remark that $S_n \subset X_n^\circ$.
  We argue that $S_n$ has Euler characteristic zero by showing that $\pi_n|_{S_n}$ is a fibration whose fiber is the union of two lines through the origin in $(\CC^*)^2$. 
  Since the fiber has Euler characteristic zero, the result follows from the fibration property.
  
  Because  $\sum_{j = 1}^{n-1} p_{jn}^2$ is a homogeneous polynomial in $x_{n}$ and $y_n$, the fiber of a point in~$\pi_n(S_n)$ is a degenerate conic in the variables $x_n, y_n$, so each fiber is the union of two lines through the origin in the complement of the lines $p_{jn}$ and conic $Q_n$ in $(\CC^*)^2$.
  Since the $p_{jn}$ all pass through the origin, they do not intersect the degenerate conic in $(\CC^*)^2$.
  The conic $Q_n$ also does not intersect $\sum_{j = 1}^{n-1} p_{jn}^2$, because  $Q_n - \sum_{j = 1}^{n-1} p_{jn}^2 = Q_{n - 1} = 0$ on the intersection, contradicting $S_n \subset X_n^\circ$.
  Since the conic $\sum_{j = 1}^{n-1} p_{jn}^2$ does not intersect the lines $p_{jn}$ or the conic $Q_n$, the fiber is the union of two lines through the origin in $(\CC^*)^2$.
\end{proof}

We do not need to compute $\chi(S_i \cap S_j)$, because the inside sum in Lemma~\ref{lem:combinatorics} vanishes for $S_i \cap S_j$.
In fact, we can count the number of critical points of the parametric log-likelihood function.
Theorem~\ref{thm:main} then follows from Proposition~\ref{prop:mainpara} and (\ref{eqn:powertwo}). 

\begin{proposition}\label{prop:mainpara}
  The parametric log-likelihood function (\ref{eqn:2parametriclikelihood}) has $2^{n-2} (n-1)!$ critical points. 
\end{proposition}

\begin{proof}
  By (\ref{eqn:powertwo}), the number of critical points is equal to the Euler characteristic $\chi(X_n)$, so it suffices to show that $\chi(X_n) = 2^{n-2} (n-1)!$.
  We proceed by induction on $n$. 
  The base case is furnished by Example~\ref{ex:basecase}.
  By Lemmas~\ref{lem:complement}, \ref{lem:combinatorics}, and \ref{lem:stratification}, we have
  \begin{align*} 
    \chi(X_{n + 1}) = \chi(X_{n + 1}^\circ) =
    \chi(F_{X_n}) \chi(X_n)
    + \sum_{i = 1}^n \chi(S_i) \sum_{S' \in \{S_i, X_n\}}
    \mu(S_i, S') (\chi(F_{S'}) - \chi(F_{X_n}) )\\
    + \sum_{1 \leq i < j \leq n} \chi(S_i\cap S_j) \sum_{S' \in \{S_i \cap S_j, S_i, S_j, X_n\}}
    \mu(S_i \cap S_j, S') (\chi(F_{S'}) - \chi(F_{X_n} ))
  \end{align*}
  By Lemma~\ref{lem:strata}, $\chi(S_i) = 0$, so the first sum is zero.
  By Lemma~\ref{lem:fibers} 
  and the values of the M\"{o}bius function in Figure~\ref{fig:posets}, the second sum is zero:
  \begin{align*}
    \sum_{S' \in \{S_i \cap S_j, S_i, S_j, X_n\}}
    \mu(S_i \cap S_j, S') (\chi(F_{S'}) - \chi(F_{X_n})) =
    (1)(-4) + (-1)(-2) + (-1)(-2) = 0. 
  \end{align*}
  By induction and Lemma~\ref{lem:fibers}, $\chi(X_{n + 1}) = \chi(F_{X_n}) \chi(X_n) = 2n(2^{n-2}(n-1)!) = 2^{n-1}n!$.
\end{proof}

A formula for the ML degree of ${\rm sGr}(d, n)$ for $d > 2$ has not yet been conjectured.
In our calculation of the ML degree of ${\rm sGr}(2,n)$, we relied heavily on the nice geometry of conics in the complex projective plane, which we do not have for higher $d$.
In \cite[Theorem 4.1]{DFRS}, the ML degrees for $d = 3$ and $n = 5, 6, 7$ were found to be $12$, $552$, and $73440$ by numerical methods.
These numbers don't follow an obvious pattern, so it is likely that the combinatorics of the stratification is complicated and that we do not have the nice cancellation we get in the proof of Proposition~\ref{prop:mainpara}.
The ML degrees of the configuration space $X(d, n)$ behave similarly.
The ML degree of $X(2, n) = \mathcal{M}_{0,n}$ is $(n - 3)!$ \cite[Proposition 1]{ST}, but when $d \geq 3$ the combinatorics of the stratification is more complicated and the Euler characteristics are more difficult to compute.
The ML degrees of $X(3,n)$ are known for $n \leq 9$ \cite[Theorem 5.1]{ABFKSTL}.
The ML degree of $X(4, 8)$ was numerically shown to be $5211816$ \cite[Theorem 6.1]{ABFKSTL}.

\section{Real Solutions}\label{sec:realsolutions}

From a statistical perspective, the only relevant critical points of the implicit log-likelihood function (\ref{eqn:implicitlikelihood}) are real, nonnegative ones.
Because each $q_I$ is a square, this condition is equivalent to the critical points of the parametric log-likelihood function (\ref{eqn:parametriclikelihood}) being real.
We will give a lower bound for all $d, n$ on the number of local maxima of the parametric log-likelihood function \eqref{eqn:parametriclikelihood}. 
Furthermore, we prove that when $d = 2$, all critical points of the parametric log-likelihood function \eqref{eqn:2parametriclikelihood} are real and local maxima. 
This is not true for larger $d$.

\begin{example}[$d=3, n = 6$]\label{ex:complexcp}
  This parametric log-likelihood function generically has $17664$ critical points \cite[Section 4]{DFRS}.
  For 100 trials with data vectors whose entries were sampled uniformly at random from $[1000]$, each parametric log-likelihood function had $11904$ real critical points, all of which were local maxima.
  These computations suggest that the parametric log-likelihood function with $d = 3, n = 6$ generically has $11904$ real solutions. 
  These computations were performed using the \verb+Julia+ package \verb+HomotopyContinuation.jl+ \cite{HC}. 
\end{example}

Let $M_{d, n}$ be a $d \times n$ matrix whose first $d \times d$ block is the identity and let $p_I$ be the maximal minor whose columns are indexed by $I$. 
For real solutions we turn from the log-likelihood function to the likelihood function
\begin{equation}\label{eqn:likelihood}
  \frac{\prod_{I \in \binom{[n]}{d}}p_I^{2u_I}} {\big (\sum_{I \in \binom{[n]}{d}}p_I^{2} \big )^{\sum_I u_I}}, 
\end{equation}
which shares its critical points with the log-likelihood function.
We optimize (\ref{eqn:likelihood}) on the real open set $X_{d,n}^\RR = X_{d,n} \cap \RR^{d(n-d)}$ where
\begin{align*}
  Q_{d, n} = \sum_{I \in \binom{[n]}{d}} p_I^2 && \textrm{and} &&
  X_{d, n} = \big \{M_{d, n} \in \CC^{d(n-d)} \colon Q_{d,n} \cdot \big ( \prod_{I \in \binom{[n]}{d}} p_I \big ) \neq 0 \big \}.
\end{align*}
Since the quadric $Q_{d,n}$ has no real solutions, $X_{d, n}^\RR$ is the complement of the real algebraic variety $\bigcup_{I \in \binom{[n]}{d}}V_\RR(p_I)$.
The irreducible hypersurfaces $V_\RR(p_I)$ all pass through the origin and divide $X_{d,n}^\RR$ into unbounded regions.
It is known that every bounded region  contains at least one critical point, because the function is either positive or negative on the region and must therefore achieve a local minimum or maximum; see \cite[Proposition 10]{CHKS}.
The following result extends this idea to unbounded regions.
\begin{lemma}\label{lem:real}
  Let $f_0, f_1, \ldots, f_n \in \RR[x_1, \ldots, x_d]$ such that $f_0 = 1$ and $f_0 + f_1 + \cdots + f_n$ is positive on $\RR^d$.
  If $u_0, u_1, \ldots, u_n \in \NN_{> 0}$, then 
  \begin{align*}
    \# \big \{\textrm{regions of $\RR^d \backslash \bigcup_{i=1}^n \{f_i = 0\}$}\big\}\\
    \leq
    \#  \{\textrm{critical points of $\frac{f_0^{u_0}f_1^{u_1} \cdots f_n^{u_n}}{(f_0 + \cdots + f_n)^{u_0 + \cdots + u_n}}$ in $\RR^d$}\}\\
    \leq
    \textrm{ML degree of $V(f_0, \ldots, f_n)$}
  \end{align*}
  Further, $L = \frac{f_0^{u_0}f_1^{u_1} \cdots f_n^{u_n}}{(f_0 + \cdots + f_n)^{u_0 + \cdots + u_n}}$ has a local maximum for every region of $\RR^d \backslash \bigcup_{i=1}^n \{f_i = 0\}$ where $L > 0$ and a local minimum for every region where $L < 0$. 
\end{lemma}
\begin{proof}
  Since $f_0 + \cdots + f_n > 0$, the function $L$ is smooth on $\RR^d$.
  Because the denominator of $L$ has larger degree than the numerator, $L$ approaches $0$ at infinity and is therefore bounded on all regions.
  Because $L$ is smooth and bounded, it attains a local maximum or minimum on each region, depending on the sign of $L$.
  This proves the first inequality and second statement.
  The second inequality follows from the definition of ML degree. 
\end{proof}

A similar technique was used to show that all critical points of the likelihood function on the moduli space $\mathcal{M}_{0,n}$ are real in \cite[Proposition 1]{ST}.
When Lemma~\ref{lem:real} is applied to our problem, we get a lower bound on the number of real critical points and on the number of local maxima.
This lower bound is given by the number of connected regions in the real open Grassmannian ${\rm Gr}_{\RR}(d, n)^\circ$.
Because the sign vectors of the Pl\"{u}cker coordinates are fixed on a given region, the number of possible sign vectors is a lower bound on the number of regions. 
We denote this set of sign vectors as ${\rm sgn}({\rm Gr}_\RR(d, n)^\circ) = \{({\rm sgn}(p_I))_{I \in \binom{[n]}{d}} \colon p \in {\rm Gr}_{\RR}(d, n)^\circ, p_{1\cdots d} = 1\}$.
This yields the following corollary:

\begin{corollary}\label{cor:lower_bounds}
  For any $d$ and $n$, the number of local maxima of \eqref{eqn:likelihood} is bounded below by the number of possible sign vectors of Pl\"{u}cker coordinates:
  \begin{align*}
    \# {\rm sgn}({\rm Gr}_\RR(d, n)^\circ) 
    \leq &\#\{\textrm{regions of $X^\RR_{d, n}$}\}\\
    &\leq
    \#\{\textrm{local maxima of \eqref{eqn:likelihood}}\} \leq \#\{\textrm{real critical points of \eqref{eqn:likelihood}}\}.
  \end{align*}
\end{corollary}
\begin{proof}
  The result follows from Lemma~\ref{lem:real} with $f_0 = p_{1\cdots d}^2 = 1$ and the fact that (\ref{eqn:likelihood}) is nonnegative on $\RR^{d(n-d)}$.
\end{proof}
When $d = 2$, the inequalities in Corollary~\ref{cor:lower_bounds} become equalities.
We prove this by showing that the number of sign vectors in ${\rm sgn}({\rm Gr}_\RR(d, n)^\circ)$ matches the number of critical points $2^{n-2}(n-1)!$ of the parametric log likelihood function \eqref{eqn:2parametriclikelihood}.
Theorem~\ref{thm:nonnegative} is an immediate corollary, since squares of nonzero real numbers are positive. 

  \begin{theorem}
    If $d = 2$, then every critical point of the parametric likelihood function \eqref{eqn:likelihood} is real and a local maximum. 
  \end{theorem}
  \begin{proof}
      We will prove that $\# {\rm sgn}({\rm Gr}_\RR(2, n)^\circ) = 2^{n-2}(n-1)!$.
  By Proposition~\ref{prop:mainpara}, the total number of critical points is $2^{n-2}(n-1)!$, so the result follows from Corollary~\ref{cor:lower_bounds}.

  To count the sign vectors, we fix $2 \leq k \leq n$  and begin with a matrix
  \begin{align*}
    M_n = \begin{pmatrix}
    1 & 0 & -x_3 & \cdots & -x_k & x_{k + 1} & \cdots & x_n\\
    1 & 0 & y_3 & \cdots & y_k & y_{k + 1} & \cdots & y_n
    \end{pmatrix}
  \end{align*}
  in $X^\RR_{2,n}$ such that $x_3,\ldots, x_n, y_3, \ldots, y_n > 0$.
  We take any permutation of the last $n - 2$ columns and flip the sign of any of the last $n - 2$ columns.
  This process yields $2^{n - 2} (n - 2)!$ different matrices.
  We argue that given a fixed $M_n$, each of these matrices has a distinct sign vector.
  Because every matrix in $X_{2,n}$ can arise in this way, there are $2^{n - 2} (n - 2)!$ possible sign vectors for a fixed $k$.
  Hence, in total, we have $2^{n - 2} (n - 1)!$ total possible sign vectors. 

  Let $p$ be the Pl\"{u}cker vector of $M_n$ and $q$ be the Pl\"{u}cker vector of a matrix produced from $M_n$ by the permutation and sign changes described above.
  We identify the permutation and columns whose signs were flipped.
  The columns $i$ with $q_{1i} < 0$ had their signs flipped. 
  Assuming that $q_{1i} > 0$ for all $i$, we uniquely identify the permutation from its inversions: for $i, j \geq 3$ the signs of $p_{ij}$ and $q_{ij}$ agree if and only if the pair $ij$ is not an inversion.
\end{proof}

The practical implication of this result is that likelihood inference for projection DPPs is difficult.
In particular, for any data $u$, the number of local maximizers of both the implicit and parametric log-likelihood function grows exponentially.
In constrast, the linear model on $\mathcal{M}_{0,n}$ in \cite{ST},  has only one positive critical point.

Example~\ref{ex:complexcp} shows that for larger $d$ the critical points are not necessarily all real. 
However, in the $d = 3, n = 6$ case, the quantities in Corollary~\ref{cor:lower_bounds} are all equal to 11904. 
We can prove this with the help of a computer: sample 1 million $3 \times 3$ matrices $A$ and compute the Pl\"{u}cker coordinates of the $3 \times 6$ matrix $[{\rm Id}_3 \,\, A]$.
The number of sign vectors we get from this process is a lower bound of 11904 on the number of sign vectors that can arise in this way.
By Example~\ref{ex:complexcp} all quantities in  Corollary~\ref{cor:lower_bounds} are therefore equal to $11904$.
\begin{conjecture}
  For $d \geq 3$, the last two inequalities in Corollary~\ref{cor:lower_bounds} are equalities, i.e.,
  \begin{align*}
    \#\{\textrm{regions of $X^\RR_{d, n}$}\}
    =
    \#\{\textrm{local maxima of \eqref{eqn:likelihood}}\} = \#\{\textrm{real critical points of \eqref{eqn:likelihood}}\}.
  \end{align*}
\end{conjecture}



\bigskip

\noindent {\bf Computational Experiments} ($d = 2$){\bf.}
  To compute the true MLE for some data $u$, one needs to compute all critical points of the parametric log-likelihood function (\ref{eqn:2parametriclikelihood}), evaluate (\ref{eqn:2parametriclikelihood}) at the critical points, and select the one which yields the largest value.
  We give runtimes for computing the MLE of our model for data selected uniformly at random from $[1000]$.
  We use the numerical algebraic geometry software \verb+HomotopyContinuation.jl+ \cite{HC} to compute the critical points of (\ref{eqn:2parametriclikelihood}).
  We use the strategy outlined in \cite[Section 3]{ST} to find solutions to the rational equations $\nabla L_u(M_n) = 0$. 
  We first use the monodromy method to compute the solutions to  $\nabla L_{u'}(M_n) = 0$ for some complex start parameters $u'$.
  We then use a coefficient parameter homotopy to move the start parameters $u'$ to the target parameters $ u \in [1000]^{\binom{n}{2}}$, simultaneously moving each solution of $\nabla L_{u'}(M_n) = 0$ to a solution of $\nabla L_u(M_n) = 0$.
  \[\begin{array}{|c|c|c|c|c|c|c|}
  \hline
                   & n = 4           & n = 5           &  n = 6 & n = 7 & n = 8& n = 9\\
  \hline
  \textrm{Runtime} & 180 \textrm{ ms} & 235 \textrm{ ms}& 962 \textrm{ ms} & 14.287 \textrm{ s} & 330 \textrm{ s} & 2 \textrm{ h}\\
  \hline
  \textrm{Number of Solutions} & 24 & 192 & 1920 & 23040 & 322560 & 5160960\\
  \hline
  \end{array}\]
  The runtime grows exponentially, as expected.
  The bulk of the time is spent computing the solutions to $\nabla L_{u'}(M_n) = 0$ with the monodromy method.
  Because \verb+monodromy_solve+ uses a heuristic stopping criterion, knowing the number of solutions a priori makes the computation significantly faster than if we did not know the number of solutions. 
  The computations were run with 64 threads on a $2 \times 8$-Core Intel Xeon Gold 6144 at 3.5 GHz.

\medskip
{\centering
  \noindent {\bf Acknowledgments}\\}

We thank Simon Telen for proving the count of critical points for $d = 2, n = 4$, which inspired the proof of Theorem~\ref{thm:main}.
We also thank Claudia Fevola, Bernhard Reinke, and Claudia Yun for helpful conversations. 
We thank the Max Planck Institute for Mathematics in the Sciences for its  stimulating research environment and for its computing resources. 

\printbibliography

\noindent Hannah Friedman, UC Berkeley
\hfill \url{hannahfriedman@berkeley.edu}

\end{document}